\documentclass[reqno, 11pt]{amsart}
\usepackage[letterpaper,hmargin=1.5in,vmargin=3cm]{geometry}
\usepackage{amsmath,amssymb,amsthm,}
\usepackage{verbatim,epsfig,graphics,hyperref}
\usepackage{marginnote}
\usepackage{color}

\newtheorem{theorem}{Theorem}[section]
\newtheorem{proposition}[theorem]{Proposition}

\newtheorem{lemma}[theorem]{Lemma}
\newtheorem{question}[theorem]{Open Question}
\newtheorem{assumption}[theorem]{Assumption}

\theoremstyle{definition}
\newtheorem{remark}[theorem]{Remark}

\newcommand{\R}{\mathbb{R}}

\newcommand{\disk}{{{\mathbb{B}}}}

\newcommand{\bt}{{\boldsymbol{t}}}
\newcommand{\bx}{{\boldsymbol{x}}}
\newcommand{\origin}{{\boldsymbol{0}}}

\DeclareMathOperator{\di}{div}
\DeclareMathOperator{\grad}{grad}
\DeclareMathOperator{\pl}{Pl}
\DeclareMathOperator{\reach}{reach}
\DeclareMathOperator{\dist}{dist}
\DeclareMathOperator{\supp}{supp}
\begin{document}
\title
{On Pleijel's nodal domain theorem for the Robin problem
}
\author{Asma Hassannezhad}
\address{University of Bristol,
School of Mathematics,
Fry Building,
Woodland Road,
Bristol, 
BS8 1UG, U.K.}
\email{asma.hassannezhad@bristol.ac.uk}
\author{David Sher}
\address{DePaul University, Department of Mathematical Sciences, 2320 N Kenmore Ave., Chicago, IL, 60614, U.S.A.}
\email{dsher@depaul.edu}
\date{}
\begin{abstract} We prove an improved Pleijel nodal domain theorem for the Robin eigenvalue problem. In particular we remove the restriction, imposed in previous work, that the Robin parameter be non-negative. We also improve the upper bound in the statement of the Pleijel theorem. In the particular example of a Euclidean ball, we calculate the explicit value of the Pleijel constant for a generic constant Robin parameter and we show that it is equal to the Pleijel constant for the Dirichlet Laplacian on a Euclidean ball. 
\end{abstract}
\maketitle
\section{Introduction}
Let $\Omega\subseteq\mathbb R^d$ be an  open, bounded, and connected domain with $C^{1,1}$ boundary, and let $h\in L^\infty(\Omega)$. We consider the eigenvalue problem for the Robin Laplacian
\begin{equation}\label{Robin problem}\begin{cases} \Delta u = \mu u & \textrm{ in }\Omega\\
 \partial_n u+h u =0 & \textrm{ on }\partial\Omega\\
\end{cases},
\end{equation}
where $\Delta=-\di\grad$
is the positive Laplacian and $n$ is the unit outward normal along $\partial\Omega$. Its spectrum consists of a discrete sequence of real numbers bounded from below, each with finite multiplicity. 
We denote its eigenvalues by $\{\mu_k\}_{k=1}^{\infty}$. Let $\{u_k\}_{k=1}^{\infty}$ be a corresponding sequence of eigenfunctions which form an orthonormal basis of $L^2(\Omega)$.  Finally, let $N_k$ be the number of nodal domains of $u_k$, which is the number of connected components of $\Omega\setminus \overline{u_k^{-1}(0)}$.  We use the notation $\Delta_{\Omega}^{R,h}$ to refer to the Robin Laplacian in \eqref{Robin problem}. By the Courant nodal domain theorem we know that $N_k\le k$ for any $k\ge 1$. 

In 1956,  Pleijel \cite{Pl56} studied the asymptotic behaviour of the number of the nodal domains of Dirichlet eigenfunctions in a planar domain. Let $\Omega\subseteq\R^2$ be open, bounded, connected, and Jordan measurable. Let $N^D_k$ denote the number of nodal domains of the eigenfunction of the Laplacian on $\Omega$ corresponding to the $k$th Dirichlet eigenvalue (with multiplicity). The Pleijel nodal domain theorem states that 
\begin{equation*}
    \limsup_{k\to\infty}\frac{N^D_k}{k}\le\gamma(2)=\frac{4}{j_0^2}.
\end{equation*}
This theorem was extended to the manifold setting in dimension two by Peetre \cite{peetre} and to any dimension by B\'erard and Meyer \cite{BM82}, with the constant $\gamma(2)$ replaced by $$
\gamma(d)=\frac{2^{d-2}d^2\Gamma(d/2)^2}{j^d_{(d-2)/2}},$$
where $j_{(d-2)/2}$ is  the smallest positive zero of the Bessel function $J_{(d-2)/2}$. Note that $\gamma(d)<1$ for $d\ge2$ \cite{BM82}. More recently Bourgain \cite{Bourgain}, Donnelly \cite{don}, and Steinerberger \cite{Stein} showed that $\gamma(d)$ is not an optimal upper bound. Their results give an improved version of the Pleijel theorem, i.e. 
\begin{equation}\label{eq:BDS}
    \limsup_{k\to\infty}\frac{N^D_k}{k}\le\gamma(d)-\epsilon(d),
\end{equation}
where $\epsilon(d)$ is a positive constant depending only on $d$. These results hold for the Dirichlet Laplacian as well as for the Laplacian on a closed manifold. 

The first result in extending the Pleijel theorem to other boundary conditions is due to I. Polterovich \cite{Pol09}, who extended the Pleijel theorem to the Neumann Laplacian  on  domains $\Omega\subset\R^2$ with  analytic boundary. L\'ena \cite{lena} extended the Pleijel theorem to the Robin problem with non-negative parameter (which includes the Neumann problem) on a bounded and $C^{1,1}$ domain~$\Omega$ in any dimension, using a different approach.

The goal of this paper is twofold. First, we prove an improved Pleijel theorem for the Robin problem \eqref{Robin problem} without any assumption on the sign of $h$.  Our result extends L\'ena's result in two ways. It removes the assumption that $h$ is non-negative. In addition, it shows that the constant in the Pleijel nodal domain theorem for the Robin problem can be improved.

Second, we calculate the exact Pleijel constant for the Robin problem on the Euclidean unit ball, with $h$ being a constant (of either sign), and show that for generic $h$ it is the same as the Pleijel constant for the Dirichlet problem on the same Euclidean unit ball. 
 
 Our first main result is the following. 
\begin{theorem}\label{Rpleijel}
There exists a positive constant  $\epsilon(d)$ depending only on $d$ such that  for any open, bounded, connected domain $\Omega\subset \R^d$ with $C^{1,1}$ boundary we have 
\begin{equation}
    \limsup_{k\to\infty}\frac{N_k}{k}\le\gamma(d)-\epsilon(d),
\end{equation}
where  $\gamma(d)$ is as in \eqref{eq:BDS}.
\end{theorem}

\begin{remark}The statement of Theorem \ref{Rpleijel} holds if $\Omega$ is an open, bounded, connected domain with $C^{1,1}$ boundary in a complete Riemannian manifold, see Theorem~\ref{Rpleijel mflds} below. \end{remark}

The proof of Theorem \ref{Rpleijel} closely follows the proof of the Pleijel nodal domain theorem for the Robin problem with positive parameter by  L\'ena \cite{lena}, incorporating the proof of the improved Pleijel theorem by Bourgain \cite{Bourgain}, Donnelly  \cite{don} and Steinerberger \cite{Stein}. In particular, the methods of the latter three authors are applied to a subset of nodal domains which L\'ena calls the \textit{bulk domains}. These are the nodal domains for which the  $L^2$-norm of the corresponding eigenfunction is concentrated away from the boundary.

The proof uses $C^{1,1}$ regularity  of the boundary, and it is open whether the Pleijel theorem holds under weaker regularity assumptions. For example, the square is not a $C^{1,1}$ domain, and the Pleijel theorem for the Robin Laplacian on the square remains open.  The only available result in this setting is the study of bounds on the number of Courant-sharp Robin eigenvalues.  A Robin eigenvalue is called \textit{Courant-sharp} if it  has a corresponding eigenfunction with exactly $k$ nodal domains; note that an immediate consequence of the Pleijel theorem for $C^{1,1}$ domains is that the number of Courant-sharp Robin eigenvlaues is finite. Gittins and Helffer  \cite{GH,GH21} studied upper bounds on the number of Courant-sharp eigenvalues of the Robin problem on a square when the Robin parameter $h$ is constant.  In particular, they show that the Robin Laplacian on a square with constant parameter $h$ has finitely many Courant-sharp eigenvalues.

We also note that a bound on the number of Courant-sharp eigenvalues of the Robin problem with non-negative parameter on domains with $C^2$ boundary  was obtained by Gittins and L\'ena \cite{GL20}. They made use of monotonicity results between Robin, Dirichlet and Neumann eigenvalues which only hold for non-negative Robin parameter.\\

The second part of this paper deals with calculating the Pleijel constant for the Robin eigenvalue problem $\Delta^{R,\sigma}_\Omega$ on a disk, where $\sigma\in \R$ is constant. We call the exact value  of $\displaystyle{\limsup_{k\to\infty}\frac{N_k}{k}}$ on a given domain $\Omega$ the \textit{Pleijel constant} and denote it by $\pl(\Delta_\Omega^{R,h})$. For the Dirichlet and Neumann Laplacian, we denote the Pleijel constant by  $\pl(\Delta^D_\Omega)$ and $\pl(\Delta^N_\Omega)$. I. Polterovich \cite{Pol09} has conjectured that
\[\sup_\Omega\pl(\Delta^D_\Omega)=\sup_\Omega\pl(\Delta^N_\Omega)= \frac{2}{\pi}=0.636\cdots,\]
where the supremum is taken over all planar domains with regular boundary (e.g. Lipschitz). 
If the conjecture is true, the inequality is sharp, as Polterovich has shown that the upper bound is attained for rectangles \cite{Pol09}.

The value of the Pleijel constant for a given domain is generally unknown. Bobkov \cite{Bobkov} calculated $\pl(\Delta_\Omega^{D})$ for some simple domains, including the Euclidean disk in $\R^2$.  When $\Omega=\disk$ is a Euclidean disk in $\R^2$, he showed that $\pl(\Delta_\disk^{D})=0.461\cdots$

We show that in the example of the unit disk, with constant Robin parameter $h=\sigma$, the value of the Pleijel constant is generically independent of the boundary condition. More precisely, we introduce the following condition on $\sigma$:
\begin{assumption}\label{assumption:bourget} The positive roots of the equation
\[\frac{zJ_{m+d/2-1}'(z)}{J_{m+d/2-1}(z)}=\frac d2+1-\sigma,\]
where $J_{\nu}(z)$ is the Bessel function, do not coincide for different values of $m\in\mathbb N_0$.
\end{assumption}
This condition should be thought of as a variant of Bourget's hypothesis \cite[p.~485]{watson}. It is generic in the sense that it may fail for at most a countable set of $\sigma\in\mathbb R$, as we show in Proposition \ref{prop:genericity}. Under this assumption we are able to prove that the Pleijel constant for the Robin problem is the same as for the Dirichlet problem.
\begin{theorem}\label{ball Pleijel constant}
Let $\disk$ be a Euclidean ball in $\R^d$. Assume that $h=\sigma$ and that Assumption \ref{assumption:bourget} holds for $\sigma$. Then $\pl(\Delta^{R,\sigma}_\disk)=\pl(\Delta^D_\disk)$. In particular for $d=2$ we have $\pl(\Delta^{R,\sigma}_\disk)=\pl(\Delta^D_\disk)=0.461\cdots$.
\end{theorem}

 The proof begins by using separation of variables to write down the Robin spectrum. Assumption \ref{assumption:bourget} shows that there is no accidental coincidence resulting in mixing of spherical harmonics from different eigenspaces, and thus each Robin eigenfunction is a spherical harmonic in the angular variables times a Bessel function in the radial variable. We show using Bessel function asymptotics that for each Robin eigenfunction, there is a corresponding Dirichlet eigenfunction with the same nodal count, and with an eigenvalue sufficiently close to the Robin eigenvalue. This is enough to prove the theorem.

It is an intriguing question to investigate the independence of the Pleijel constant from the boundary condition. More precisely

\begin{question}
Can we show that $\pl(\Delta^{R,h}_\Omega)=\pl(\Delta^D_\Omega)$ for any bounded connected domain $\Omega$ with regular enough boundary (e.g. Lipschitz or $C^{1,1}$ boundary)?
\end{question}

\section{Proof of Theorem \ref{Rpleijel}}

The proof of Theorem \ref{Rpleijel} consists of two parts: the first part relies on the proof of the Pleijel theorem for the Robin problem with non-negative parameter by  L\'ena \cite{lena}, and the second part on the proof of the improved Pleijel theorem by Bourgain \cite{Bourgain}, Donnelly  \cite{don} and Steinerberger \cite{Stein}. Let us first review the main ideas of the previous work.

The classical proof of the Pleijel theorem for the Dirichlet Laplacian uses only the Faber-Krahn inequality (for all nodal domains) and the Weyl law.  However, for the Neumann and Robin eigenvalue problems,  the Faber-Krahn inequality does not hold for the  nodal domains which are adjacent to the boundary. For bounded planar domains with analytic boundary,  I.~Polterovich obtained the Pleijel nodal domain theorem  using an estimate by Toth and Zelditch \cite{TZ09} that the growth rate of the number of nodal domains adjacent to the boundary is smaller than the index $k$.  L\'ena extended the Pleijel theorem to the Robin problem with non-negative parameter on a bounded domain $\Omega\subset\R^n$ with $C^{1,1}$ boundary using a completely different approach. 

Since our proof relies on L\'ena's proof, we briefly explain the main ideas. L\'ena considered two families of nodal domains, which he called the bulk and boundary nodal domains.
A nodal domain $D$ for which  the  $L^2$-norm of the eigenfunction on $D$ concentrates away from the boundary  of $\Omega$ is called a \textit{bulk nodal domain}, while one for which the $L^2$-norm of the eigenfunction on $D$ has significant mass near the boundary of $\Omega$ is referred to as a \textit{boundary nodal domain}. He then showed that the growth rate of the number of boundary nodal domains is lower than $k$. Thus the boundary nodal domains do not contribute to the limit of the fraction $N_k/k$ as $k\to\infty$. We use the same decomposition of nodal domains into bulk and boundary domains and extend L\'ena's result \emph{without} imposing any constraint on the Robin parameter $h$. Roughly speaking, we show that the contribution from $h^{-}=\max\{-h,0\}$ contributes to the lower order terms but disappears when divided by $k$ in the limit.

To get an improved version of the Pleijel theorem for the Laplacian, the main idea (as in \cite{Bourgain,don,Stein}) is to follow the original proof of the Pleijel theorem but use a quantitative version of the Faber-Krahn inequality and show that it leads to an improvement of the upper bound in the Pleijel theorem.  We adapt this idea for the Robin eigenvalue problem by applying  a quantitative version of the Faber-Krahn inequality for the bulk domains only. \\

We begin with the same general approach as outlined in  \cite{lena}. We first give  two key propositions concerning the eigenfunctions. The first involves boundary regularity, generalizing \cite[Proposition 1.6]{lena}. Throughout this section, we will mainly use the same notations as in  \cite{lena}.

\begin{proposition}\label{prop:regularity}
Any eigenfunction $u$ of $\Delta_{\Omega}^{R,h}$ is an element of $C^1(\overline{\Omega})$.
\end{proposition}

\begin{proof}
The statement follows from some elliptic regularity arguments. The same line of argument is used  to obtain the same statement for Steklov eigenfuctions by Decio \cite{Decio}, and   for Robin eigenfunctions with positive parameter by L\'ena \cite{lena}. 

Specifically, let $u$ be an eigenfunction of $\Delta_{\Omega}^{R,h}$ with eigenvalue $\mu$. Then $u\in H^1(\Omega)$ (see e.g. \cite{daners09}), and its trace is in $H^{1/2}(\Omega)$. By \cite[Proposition 2.5.2.3]{Grisvard}, if $v$ is a weak solution of 
\[\begin{cases}
\Delta v +v=f& \text{in $\Omega$}\\
\partial_nv=g&\text{on $\partial\Omega$}
\end{cases}\]
with $f\in L^2(\Omega)$ and $g\in W^{1-1/p,p}(\Omega)$, then $v\in W^{2,p}(\Omega)$. Taking  $p=2$, $v=u$,$f=(1+\mu)u$, and $g=-hu$, we conclude that $u\in W^{2,2}(\Omega)$. We  get $u\in W^{2,p}(\Omega)$ for all $p<\infty$ by  bootstrapping the argument, and conclude that $u\in C^{1}(\bar\Omega)$ by the Sobolev embedding theorem.  We refer the reader to the proof of \cite[Proposition 2]{Decio} for details.
\end{proof}

The second is a replacement for Green's formula, generalizing \cite[Propositions 1.7 and 4.3]{lena}. 

\begin{proposition}\label{prop:green}
Let $H=\|h^-\|_\infty$ where $h^-(x)=\max\{-h(x),0\}$. There exists a constant $C$ depending only on $\Omega$ such that if $u$ is any eigenfunction of $\Delta_{\Omega}^{R,h}$ with eigenvalue $\mu$, and $D$ is any nodal domain of $u$, then
\[\int_D |\nabla u|^2\, dx \leq (\sqrt{\mu}+CH+\sqrt{CH})^2\int_D u^2\, dx.\]
\end{proposition}

\begin{proof}quality \eqref{eq:beforenew} is the same as the proof of \cite[Propositions 1.7 and 4.3]{lena}. Let $\tilde u: U\to \R$ be a $C^1$-extension of $u$ to an open neighborhood $U$ of $\bar \Omega$. Since $u\in C^{1}(\bar \Omega)$ by Proposition \ref{prop:regularity}, such an extension exists.  Let $\tilde D$ be the nodal domain of $\tilde u$ containing $D$.  For $\alpha>0$, we denote 
\[D_{\alpha}=D\cap\{u(x)>\alpha\},\qquad \tilde D_\alpha:=D\cap\{\tilde u(x)>\alpha\},\]
\[ \Gamma_\alpha:=\tilde D_\alpha\cap \partial \Omega,\quad \Sigma_\alpha:=\partial \tilde D_\alpha\cap\Omega,\quad
\gamma_\alpha:=\partial \tilde D_\alpha\cap\partial\Omega.\]
Let $\alpha>0$ be such that $\alpha$ is a regular value of both $\tilde u$ and $\tilde u|_{\partial\Omega}$. This ensures that $\partial D_\alpha=\Gamma_\alpha\cup\Sigma_\alpha\cup\gamma_\alpha$ is Lipschitz and $\gamma_\alpha$ is an $(n-2)-$dimensional submanifold of $\partial \tilde D_
\alpha$ (see \cite{lena} for details). Note that by Sard's theorem there exists an infinite sequence of such regular values approaching zero. 
 We can now apply Green's theorem to the function $u_{\alpha}=u-\alpha$ in $D_{\alpha}$ and obtain
\[\int_{D_{\alpha}}|\nabla u|^2\, dx = \int_{D_{\alpha}}{\Delta u_{\alpha}}u_{\alpha} - \int_{\Gamma_{\alpha}}u_{\alpha}\frac{\partial u_{\alpha}}{\partial n}\, ds.\]
Plugging in the information we know,
\[\int_{D_{\alpha}}|\nabla u|^2\, dx = \mu\int_{D_{\alpha}}uu_{\alpha}\, dx - \int_{\Gamma_{\alpha}}huu_{\alpha}\, ds.\]
Since $u$ and $u_{\alpha}$ are both positive on $\Gamma_{\alpha}$ and on $D_{\alpha}$,  
we conclude that
\begin{equation}\label{eq:beforenew}
\int_{D_{\alpha}}|\nabla u|^2\, dx \leq \mu\int_{D_{\alpha}}u^2\, dx + H\int_{\Gamma_{\alpha}}u^2\, ds.\end{equation}

\emph{Step 2.~} Let us bound the second term in the right-hand side of inequality \eqref{eq:beforenew} in terms of the $L^2$-norms of $u$ and $\nabla u$. Let $\rho(x)$ be a function on $\Omega$ which agrees with $\dist(x,\partial\Omega)$ in a neighborhood of the boundary $\partial\Omega$.  By Lemma \ref{distance function}, below, we have that $\dist(x,\partial\Omega)$ is $C^{1,1}$ on a neighbourhood of $\partial\Omega$. Thus, we can assume that $\rho\in C^{1,1}(\Omega)$.
Then by Green's theorem,
\[\int_{\Gamma_{\alpha}}u^2\, ds = \int_{\partial D_{\alpha}}u^2\partial_{\nu}\rho\, ds = \int_{D_{\alpha}}\nabla(u^2)\cdot \nabla\rho\, dx {-}\int_{D_{\alpha}}u^2\Delta\rho\, dx,\]
and thus
\[\int_{\Gamma_{\alpha}}u^2\, ds \leq \int_{D_{\alpha}}u^2|\Delta\rho|\, dx + \int_{D_{\alpha}}2|u|\, |\nabla u|\, |\nabla\rho|\, dx.\]
Since $\rho$ depends only on $\Omega$ and is in {$C^{1,1}(\Omega)$}, there exists a constant $C$ such that
\[\int_{\Gamma_{\alpha}}u^2\, ds \leq C\int_{D_{\alpha}}u^2\, dx + C\int_{D_{\alpha}}|u|\, |\nabla u|\, dx.\]
By the Cauchy-Schwarz inequality,
\[\int_{\Gamma_{\alpha}}u^2\, ds \leq C\int_{D_{\alpha}}u^2\, dx + C\left(\int_{D_{\alpha}}u^2\, dx\right)^{1/2}\left(\int_{D_{\alpha}}|\nabla u|^2\, dx\right)^{1/2}.\]
Plugging this into \eqref{eq:beforenew} and dividing through by $\int_{D_{\alpha}}u^2\, dx$, we get
\[\frac{\int_{D_{\alpha}}|\nabla u|^2\, dx}{\int_{D_{\alpha}}u^2\, dx}\leq \mu + CH + CH\Big(\frac{\int_{D_{\alpha}}|\nabla u|^2\, dx}{\int_{D_{\alpha}}u^2\, dx}\Big)^{1/2}.\]
Let us denote
\[w:=\frac{\int_{D_{\alpha}}|\nabla u|^2\, dx}{\int_{D_{\alpha}}u^2\, dx}.\]
It is easy to see that $w\leq A+B\sqrt{w}$ implies
\[w\leq\left(\frac B2+\sqrt{\left(\frac B2\right)^2 + A}\, \right)^2\leq\left(B+\sqrt{A}\right)^2.\]

Here, $A=\mu+CH$ and $B=CH$, so $B+\sqrt{A}\leq CH+\sqrt{CH}+\sqrt{\mu}$ and thus
\[\frac{\int_{D_{\alpha}}|\nabla u|^2\, dx}{\int_{D_{\alpha}}u^2\, dx}\leq(\sqrt{\mu}+CH+\sqrt{CH})^2.\]
The result of Proposition \ref{prop:green} now follows by taking the limit as $\alpha$ tends to zero along the sequence of regular values guaranteed by Sard's theorem. Note that both $u$ and $|\nabla u|$ are continuous on $\overline{\Omega}$ by Proposition \ref{prop:regularity}.
\end{proof}

We now prove the regularity of the distance function on a neighborhood of $\partial\Omega$ which is used in the second step of the proof of Proposition \ref{prop:green} above.

We say $\partial\Omega$ has a \textit{positive reach} if there exists $\delta>0$ such that for any  $x\in{\partial\Omega}_\delta:=\{y\in \R^n: \dist(y,\partial\Omega)<\delta\}$, there exists a unique nearest point in $\partial\Omega$. We denote the supremum of such $\delta$ by $\reach(\partial \Omega)$. We refer to \cite[page 432]{Fed59} for more details on sets with positive reach. 
\begin{lemma}\label{distance function}
    Let $\Omega\subset\R^n$ be an open bounded connected domain with $C^{1,1}$ boundary. Then $\partial\Omega$ has a positive reach  and $\dist(x,\partial\Omega)$ is $C^{1,1}$ on a neighbourhood of $\partial\Omega$.
\end{lemma}
In \cite{KP81}, Krantz and Parks showed that if $\partial\Omega$  is $C^1$ and has a positive reach then  $\dist(x,\partial\Omega)$ is also $C^1$. Note that assuming $\partial\Omega$ is $C^1$ is not enough to conclude that it has a positive reach, see \cite[Example 4]{KP81}. The proof of  Lemma \ref{distance function} follows the same idea of the proof of \cite[Theorem 2]{KP81} together with the results of \cite[Section 4]{Fed59}. 
\begin{proof}[Proof of Lemma \ref{distance function}] The fact that $\reach(\partial\Omega)>0$ is a consequence of \cite[Theorem 4.12]{Fed59}. For the regularity of the distance function, we repeat the same argument as in the proof of \cite[Theorem 2]{KP81}. Namely, for any $p\in \partial\Omega$,  we choose a coordinate system putting $p$ at the origin such that in a neighbourhood $U$ of $p=\origin$ we have $$U\cap\partial \Omega=\left\{(\bt,f(\bt)): \bt\in B_r^{n-1}(\origin)\subset\R^{n-1}, f\in C^{1,1}( B_r^{n-1}(\origin))\right\}$$ with $f(\origin)=0$ and $\nabla f(\origin)=\origin$. Let $\delta\in (0,\reach(\partial\Omega))$ be sufficiently small so that for every $(\bx,y)\in B_\delta^n(\origin)$, there exists a unique nearest point $(\bt,f(\bt))\in \partial\Omega$ with $\bt\in B_r^{n-1}(\origin)$. Since $\bt(\bx,y)$ is an extremum of $|\bt-\bx|^2+|f(\bt)- y|^2$. We have 
\begin{equation}\label{eq1}(\bt-\bx)-(f(\bt)-y))\nabla f(\bt)=0.\end{equation}
Let $d(\bx,y):=\dist((\bx,y),\partial\Omega)=\left(|\bt-\bx|^2+|f(\bt)- y|^2\right)^{1/2}$. By \cite[Theorem 4.8(5)]{Fed59}, $d$ is differentiable on $B^{n}_\delta(\origin)\setminus\partial\Omega$ and
\begin{equation}\label{eq2}\nabla d=\frac{1}{d}(\bx-t,f(\bt)-y),\qquad (\bx,y)\in B_\delta^{n}(\origin)\setminus\partial\Omega.\end{equation}
Using \eqref{eq1}, we can write $d(\bx,y)=|f(\bt)-y)|(1+|\nabla f (\bt)|^2)^{1/2}.$ By plugging this into \eqref{eq2} and using \eqref{eq1}, we get
\begin{equation}\label{eq3}\nabla d=\frac{|(1+|\nabla f (\bt)|^2)^{-1/2}}{|y-f(\bt)|}((f(\bt)-y)\nabla f(\bt),f(\bt)-y),\qquad (\bx,y)\in B_\delta^{n}(\origin)\setminus\partial\Omega.\end{equation}
The right-hand side of \eqref{eq3} is $C^{0,1}$. Therefore by \cite[Theorem 4.7]{Fed59} we conclude that the above identity holds on $B_\delta^{n}(\origin)$, and thus that $d$ is $C^{1,1}$ on $B_{\delta}^{n}(\origin)$. Since $\partial\Omega$ is compact, the result follows by taking a finite open cover.
\end{proof}
\begin{remark}\label{key props}
  Although Propositions \ref{prop:genericity} and \ref{prop:green} are stated in the Euclidean setting, they continue to hold for a bounded $C^{1,1}$ domain in a complete Riemannian manifold.  Indeed the argument showing the  regularity of the eigenfunctions in Proposition~\ref{prop:regularity} and the Green's formula used in the proof of Proposition \ref{prop:green} clearly extend to the Riemannian setting. Moreover, the proof of the regularity of $\rho$ is verbatim as in the proof of Lemma \ref{distance function} if $\partial\Omega$ has a positive reach, which itself follows from   \cite[Theorem 4.12]{Fed59} by considering a finite cover of $\partial\Omega$. 
\end{remark}
To proceed with the proof of Theorem \ref{Rpleijel}, we need to recall the definitions of bulk and boundary nodal domains from \cite{lena}. Let $\partial\Omega_\delta$ denote the $\delta$-neighbourhood of $\partial\Omega$ in $\Omega$, 
\[\partial\Omega_\delta:=\{x\in\Omega: \dist(x,\partial\Omega)<\delta\}.\]

Then by \cite[Lemma 2.1]{lena}, for any sufficiently small $\delta>0$, there exist constants $0<c<C$ and functions
$\phi^\delta_0,\phi^\delta_1\in C^{1,1}(\Omega)$ such that 
$\supp(\phi^\delta_0)\subset \Omega\setminus\partial\Omega_{c\delta},$  $\supp(\phi^\delta_1)\subset \partial\Omega_{C\delta},$ and $(\phi^\delta_0)^2+(\phi^\delta_1)^2\equiv1$.

Let $u$ be an eigenfunction of $\Delta_{\Omega}^{R,h}$ and $D$ a nodal domain of $u$. We  consider the two functions $u_0:=\phi^\delta_0u$ and $u_1:=\phi_1^\delta u$. For a given $\epsilon>0$, we say that $D$ is  an \emph{{$\epsilon$-bulk} nodal domain} if  
 \begin{equation}\label{bulkboundary}\int_{D}u_0^2\,{\rm d}x\ge (1-\epsilon)\int_{D}u^2{\rm d}x,\end{equation}
 and $D$ is an  \emph{{$\epsilon$-boundary} nodal domain} if 
 \[\int_{D}u_1^2\,{\rm d}x> \epsilon\int_{D}u^2{\rm d}x.\]

\begin{proof}[Proof of Theorem \ref{Rpleijel}]
Let $u$ be an eigenfunction with eigenvalue $\mu_k$, and let $u_0=\phi^\delta_0u$ and $u_1=\phi_1^\delta u$, with $\delta=\mu_k^{-\frac{1}{4}}$. Fix a sufficiently small $\epsilon>0$, and let $\{D_j\}_{j=1}^{N_B(k,\epsilon)}$ be the family of {$\epsilon$-bulk} nodal domains of $u$,
and $\{D_j^b\}_{j=1}^{N_b(k,\epsilon)}$ the family of {$\epsilon$-boundary} nodal domains of $u$.
 Let $B_k$ be a ball with first Dirichlet eigenvalue $\lambda^D_1(B_k)=\mu_k$.  For given $\epsilon_1,\epsilon_2>0$ to be determined later,  we categorise the  nodal domains into three disjoint sets as in \cite{don,Bourgain}. 
\begin{itemize}\label{donellyargument}
    \item[I.] $|D_j|\ge |B_k|(1+\epsilon_1)$;
    \item[II.]$|D_j|\le |B_k|(1+\epsilon_1)$ and $A(D_j)\ge\epsilon_2$;
    \item[III.]$|D_j|\le |B_k|(1+\epsilon_1)$ and $A(D_j)<\epsilon_2$,
\end{itemize}
where $A(D)$ is the Fraenkel asymmetry and $|\cdot|$ denotes the $n-$volume of the set. We refer to the proof of Claim 2 below where the choices of $\epsilon_1$ and $\epsilon_2$ are discussed. Crucially, we emphasise that their values are independent of $\epsilon$.  \\Recall that 
\[A(D)=\inf\left\{\frac{|D\Delta B|}{|B|}:\text{ $B$ is a ball with $|B|=|D|$}\right\}.\]
Let $N_{\star}$ be the number of {$\epsilon$-bulk} nodal domains satisfying condition $\star={\rm I,II}$ or II. Note that $N_B(k,\epsilon)=N_{\rm I}+N_{\rm II}+N_{\rm III}$.\\

For any $\epsilon$-bulk nodal domain $D$, we have the following upper bound for $\lambda_1^D(D)$ calculated  in \cite[page 291]{lena}:
\begin{equation*}
\lambda_1^D(D_j)\le \frac{\int_{D_j}|\nabla u_0|^2}{\int_{D_j}u_0^2}
\le\frac{1+\epsilon}{1-\epsilon}\frac{\int_{D_j}|\nabla u|^2}{\int_{D_j}u^2}+C_0\frac{1+\frac{1}{\epsilon}}{1-\epsilon}\sqrt{\mu_k},
\end{equation*}
where $C_0$ is  a constant  independent of $\mu_k$. Applying  Proposition \ref{prop:green} we get
\begin{equation}\label{inq:lena}
\lambda_1^D(D_j)\le \frac{\int_{D_j}|\nabla u_0|^2}{\int_{D_j}u_0^2}
\le\frac{1+\epsilon}{1-\epsilon}{\mu_k}+C_1\sqrt{\mu_k}+C_2,
\end{equation}
where $C_1=2(CH+\sqrt{CH})+C_0\frac{1+\frac{1}{\epsilon}}{1-\epsilon}$ and $C_2=(CH+\sqrt{CH})^2$. \\

 \noindent  We first estimate the number {$\epsilon$-bulk} nodal domains in each category I, II, III. \\

\noindent \textit{Case I.} For any {$\epsilon$-bulk} nodal domain $D_j$ of type I, we trivially have 
\begin{equation}\label{inq:1}
\omega_n j^n_{\frac{n-2}{2}}(1+\epsilon_1)=\lambda_1^D(B_k)^{n/2}{|B_k|}(1+\epsilon_1)\le \mu_k^{n/2}|D_j|.
\end{equation}
After summing over all  {$\epsilon$-bulk} nodal domains of type I and rearranging inequality \eqref{inq:1}, we get
\begin{eqnarray}
N_I\le \mu_k^{\frac{n}{2}}\frac{|\Omega_I|}{\omega_n j^n_{\frac{n-2}{2}}(1+\epsilon_1)},
\end{eqnarray}
where $\Omega_I$ is the union of all  {$\epsilon$-bulk} nodal domains of type I. \\

\noindent\textit{Case II.}
For any $\epsilon$-bulk nodal domain $D_j$ of type II,  we use  the quantitative Faber-Krahn inequality~\cite{Brascoetal} to get 
\begin{equation}\label{inq:QFK}
\lambda_1^D(B_k)^{n/2}{|B_k|}(1+\epsilon_3)\le\lambda_1^D(B_k)^{n/2}{|B_k|}(1+c(n)A(D_j)^2)\le \lambda_1^D(D_j)^{n/2}|D_j|,
\end{equation}
where $\epsilon_3=c(n)\epsilon_2^2$. We  use inequality \eqref{inq:lena} and sum over all type II {$\epsilon$-bulk} nodal domains to get
 \begin{eqnarray}
N_{\rm II}\le \left( \frac{1+\epsilon}{1-\epsilon}{\mu_k}+ C_1\sqrt{\mu_k}+C_2\right)^{\frac{n}{2}}\frac{|\Omega_{\rm II}|}{\omega_n j^n_{\frac{n-2}{2}}(1+\epsilon_3)},
\end{eqnarray}
where $\Omega_{\rm II}$ is the union  {$\epsilon$-bulk} nodal domains of type II. \\

\noindent \textit{Case III.} For any $\epsilon$-bulk nodal domain $D_j$ of type III,  the Faber-Krahn inequality implies that 
\begin{equation*}
\lambda_1^D(B_k)^{n/2}{|B_k|}\le \lambda_1^D(D_j)^{n/2}|D_j|,
\end{equation*}
We use inequality \eqref{inq:lena} and sum over all $\epsilon$- bulk nodal domains of type~III  to get
\begin{eqnarray}
N_{\rm III}\le \left( \frac{1+\epsilon}{1-\epsilon}{\mu_k}+ C_1\sqrt{\mu_k}+C_2\right)^{\frac{n}{2}}\frac{|\Omega_{\rm III}|}{\omega_n j^n_{\frac{n-2}{2}}},
\end{eqnarray}
where $\Omega_{\rm III}$ is the union  {$\epsilon$-bulk} nodal domains of type III. \\

Now let us recall the  Weyl asymptotic for the Robin problem (see e.g. \cite{BS80,FG12}):
\begin{equation}\label{weyl}
\lim_{k\to\infty}\frac{\mu_k^{\frac{n}{2}}|\Omega|}{k}=\frac{(2\pi)^n}{\omega_n},
\end{equation}
and the definition of $\gamma(n)$, $$\gamma(n):=\frac{(2\pi)^n}{\omega_n^2 j^n_{\frac{n-2}{2}}}.$$ 
Hence, \label{donnellyargument2}
\begin{eqnarray*}
   \limsup_{k\to\infty}\frac{N_B(k,\epsilon)}{k}&=& \limsup_{k\to\infty}\frac{N_{\rm I}+N_{\rm II}+N_{\rm III}}{k}\\
&\le&\gamma(n)\left(\frac{\alpha_{\rm I}}{1+\epsilon_{\rm 1}}+\left( \frac{1+\epsilon}{1-\epsilon}\right)^{\frac{n}{2}}\left(\frac{\alpha_{\rm II}}{1+\epsilon_3}+\alpha_{\rm III}\right)\right),
\end{eqnarray*}
where $\alpha_\star=\frac{|\Omega_{\star}|}{|\Omega|}$,  $\star={\rm I, II, III}$.\\

\noindent To conclude the proof, we need to prove the following statements.\\

\noindent\textit{Claim 1.} For the $\epsilon$-boundary nodal domains we have 
\[\limsup_{k\to\infty}\frac{N_b(k,\epsilon)}{k}=0.\]
\noindent\textit{Proof of Claim 1.~} The proof follows from \cite[Section 2.4]{lena} and we do not repeat the details of the calculation. The main idea is to extend $u_1$ by reflection to   a function $v$ on a neighbourhood of each {$\epsilon$-boundary} nodal domain and apply  the Faber-Krahn inequality on each connected component of the interior of the  support of $v$. Here, the regularity of the eigenfunction discussed in  Proposition \ref{prop:regularity} is needed. Then we can get an estimate for $N_b(k,\epsilon)$:
\begin{equation}\label{boundaryterm}N_b(k,\epsilon)\le C_1\alpha_k^{-\frac{1}{4}}\epsilon^{-\frac{n}{2}}\left (\alpha_k+C_2\alpha_k^{\frac{1}{2}}\right)^{\frac{n}{2}},
\end{equation}
where $\alpha_k=(\sqrt{\mu_k}+CH+\sqrt{CH})^2$.
For any fixed $\epsilon>0$, the leading term in the right-hand side of \eqref{boundaryterm} goes to zero as $k\to \infty$ because of the Weyl law \eqref{weyl}, completing the proof of Claim 1.  \qed\\

\noindent\textit{Claim 2.} {There exist suitable choices} of {$\epsilon_0$,} $\epsilon_1$, and $\epsilon_2$ { such that for any $\epsilon \in(0,\epsilon_0)$,} the proportion of {$\epsilon$-bulk} nodal domains of type III {corresponding to $\mu_k$, $k\ge k(\epsilon)$} cannot be one, i.e. $\alpha_{\rm III}:=\frac{|\Omega_{\rm III}|}{|\Omega|}\le c(n)<1$. \\

\noindent\textit{Proof of Claim 2.~} The proof of the claim follows by adapting the argument given in  \cite[Page 59]{don}.  Recall that for any nodal domain  $D$ of type III,  we have 
\begin{equation}\label{typeIII}|D|\le(1+\epsilon_1)|B_k|, \quad\text{and}\quad A(D)< \epsilon_2.\end{equation} The second inequality means that there exists an approximate ball $B$ such that $|D|=|B|$ and $|D\Delta B|< \epsilon_2|B|$.

We now show that for any pair of  type III nodal domains $D$ and $\tilde D$,  their approximate balls $B$ and  $\tilde B$  almost have the same radius and their overlap can be made small if we choose $\epsilon$, $\epsilon_1$, and $\epsilon_2$  small enough.

By the Faber-Krahn inequality we have 
\begin{equation}\label{eq:fk}|B_k|\le \left(\frac{\lambda_1^D(D)}{\lambda_1^D(B_k)}\right)^{\frac{n}{2}}|D|=\left(\frac{\lambda_1^D(D)}{\mu_k}\right)^{\frac{n}{2}}|D|.\end{equation}
For any fixed $\epsilon$, we use the estimate in \eqref{inq:lena}.  It implies that for $k\ge k(\epsilon)$
we can have 
\begin{equation*}\label{eq:FKestimate}
    \frac{\lambda_1^D(D)}{\mu_k}\le \frac{1+2\epsilon}{1-\epsilon}.
\end{equation*}
Putting \eqref{typeIII} and \eqref{eq:fk} together, we get
\begin{equation}
\left(\frac{1-\epsilon}{1+2\epsilon}\right)^{\frac{n}{2}}|B_k|\le |D|\le (1+\epsilon_1)|B_k|
\end{equation}
Thus,
\[\left(\frac{1-\epsilon}{1+2\epsilon}\right)^{\frac{n}{2}}(1+\epsilon_1)^{-1}\le\frac{|\tilde D|}{|D|}=\frac{|\tilde B|}{|B|}\le (1+\epsilon_1)\left(\frac{1+2\epsilon}{1-\epsilon}\right)^{\frac{n}{2}}. \]
Hence, if we choose $\epsilon, \epsilon_1>0$ small enough, it guarantees that the approximate balls  almost have the same radius. 
Moreover, we have $|B\cap \tilde B|\le |D\Delta B|+|\tilde D\Delta \tilde B|<\epsilon_2(|B|+|\tilde B|)=\epsilon_2(|D|+|\tilde D|)$, and by a choice of $\epsilon_2$ small enough, their overlap can be made small. Note that the $\epsilon_1$ and $\epsilon_2$  can be chosen independent of $\epsilon$.

We can complete the argument by showing that the proportion of $\Omega$ covered by type III nodal domains is uniformly bounded away from  1. By the result of Rogers \cite{Rog58}, the  packing density $\rho(n)$  of  balls of the same radius in $\R^n$ is strictly less than one. Roughly speaking, we can obtain a packing for the domain $\Omega$ by slightly reducing the size of the balls. The ratio $\alpha_{III}$ of the volume of type III nodal domains to the total volume is then approximately equal to $(1+o(1))\rho(n)$. See \cite[Page 59]{don} for more details. \qed\\

Therefore, 
\begin{eqnarray*}
\limsup_{k\to\infty}\frac{N_k}{k}&=&\limsup_{k\to\infty}\frac{N_B(k,\epsilon)}{k}\\
&\le&\gamma(n)\left(\frac{\alpha_{\rm I}}{1+\epsilon_{\rm 1}}+\left( \frac{1+\epsilon}{1-\epsilon}\right)^{\frac{n}{2}}\left(\frac{\alpha_{\rm II}}{1+\epsilon_3}+\alpha_{\rm III}\right)\right).
\end{eqnarray*}
By Claim 2, we have $\alpha_{\rm III}\le c(n)<1$. Note that $\epsilon_3$ does not depend on $\epsilon$.  The statement of Theorem \ref{Rpleijel} follows by taking $\epsilon\to 0$. 
\end{proof}

\begin{theorem}\label{Rpleijel mflds}
The statement of Theorem \ref{Rpleijel} holds for $\Omega$ a bounded open domain with $C^{1,1}$ boundary in a complete Riemannian manifold.  
\end{theorem}
\begin{proof} It is enough to show that the key ingredients of the proof in the Euclidean setting can be adapted to the Riemannian setting. From Remark \ref{key props}, we know that the Propositions \ref{prop:regularity} and \ref{prop:green} hold in the Riemannian setting. 
We consider the decomposition of nodal domains into {$\epsilon$-bulk} and {$\epsilon$-boundary} nodal domains as defined in  \eqref{bulkboundary} and adapt the argument on pages \pageref{donellyargument}-\pageref{donnellyargument2} as in \cite[Section 3]{don}. The main idea is to use a weaker version of the Faber-Krahn inequality and its quantitative version, which hold only for {$\epsilon$-bulk} nodal domains with volume smaller than a given threshold number. The number of {$\epsilon$-bulk} nodal domains {with volume larger than the threshold number}
is bounded independent of $k$ and so does not contribute to the limit. Hence we can repeat the same argument as in the proof of Theorem \ref{Rpleijel}. We refer the reader to \cite[Section 3]{don} for more details. 

\end{proof}

\section{Pleijel constant for the Robin problem on a Euclidean ball}

Throughout this section we will consider the Robin problem on a unit ball in $\mathbb R^d$, $d\geq 2$.  We assume throughout that $h=\sigma$ is a constant, so our boundary condition is
\[\partial_{n}u+\sigma u =0.\]
 Note that $\sigma\ge 0$ yields non-negative spectrum, but we need not assume this.

We begin by writing down the spectrum of this problem for arbitrary $\sigma\in\mathbb R$. To start, consider the Dirichlet spectrum of the ball in $\mathbb R^d$. As it is well known, it consists of the squares of zeroes of Bessel functions, specifically
\[\lambda_{m,k} = j_{\nu,k}^2,\quad m\ge 0, k\in\mathbb N\]
where $\nu=m+d/2-1$ and $j_{\nu,k}$ is the $k$th positive zero of $J_{\nu}(z)$. These eigenvalues have multiplicity $\kappa_{m,d}$, where
\[\kappa_{0,d}=1,\quad\textrm{ otherwise } \kappa_{m,d}=\binom{m+d-1}{d-1}-\binom{m+d-3}{d-1}.\]
Note that when $d=2$, $\kappa_{m,d}=2$ for all $m\in\mathbb N$. The multiplicity $\kappa_{m,d}$ is simply the multiplicity of the space of spherical harmonics with eigenvalue $m(m+d-2)$.

We now consider the Robin spectrum when $\sigma\geq 0$. In this case, the eigenvalues are all non-negative. They are
\[\mu_{m,k} = x_{m+d/2-1,k}^2,\quad m\ge 0,~ k\in\mathbb N\]
again with multiplicity $\kappa_{m,d}$. However, the values $x_{\nu,k}$ are not just Bessel function zeroes but are the $k$th non-negative roots of the equation
\[\frac{zJ_{\nu}'(z)}{J_{\nu}(z)} = \frac d2-1-\sigma.\]
There is  exactly one such $x_{\nu,k}$ in between each pair of roots of $J_{\nu}(k)$ \cite[pp. 480-481]{watson}. Furthermore, $x_{\nu,1}< j_{\nu,1}$.

For $\sigma<0$, there are finitely many negative Robin eigenvalues. All of the values $x_{\nu,k}^2$ are still eigenvalues, regardless of the sign of $\sigma$. And for $m\geq -\sigma$, we still have $0\le x_{\nu,1}< j_{\nu,1}$. However, when $m<-\sigma$, this no longer holds and instead $x_{\nu,1}>j_{\nu,1}$. The reason is that there is now a negative eigenvalue, specifically $-(x_{\nu,0})^2$, where $x_{\nu,0}$ is the unique positive root of the equation
\[\frac{zI_{\nu}'(z)}{I_{\nu}(z)}=\frac d2-1-\sigma,\]
with $I_{\nu}(z)$ the modified Bessel function. This eigenvalue has multiplicity $\kappa_{m,d}$ as before. Thus the number of negative Robin eigenvalues, with multiplicity, is
\[N_{\sigma}:=\sum_{m=0}^{\lceil -\sigma\rceil}\kappa_{m,d}.\]

To summarize, for $m\geq -\sigma$, each $x_{\nu,k}^2$ is an eigenvalue, and
\begin{equation}\label{eq:interlacinglarge}0<x_{\nu,1}<j_{\nu,1}<\dots<x_{\nu,k}<j_{\nu,k}<x_{\nu,k+1}<\dots
\end{equation}
However, for $m<-\sigma$, each $x_{\nu,k}^2$ is an eigenvalue, but there is also a corresponding negative eigenvalue $-x_{\nu,0}^2$, and we have instead that
\begin{equation}\label{eq:interlacingsmall}0<j_{\nu,1}<x_{\nu,1}<\dots<x_{\nu,k}<j_{\nu,k+1}<x_{\nu,k+1}<\dots\end{equation}

When enumerating both Dirichlet and Robin spectra we have to put the eigenvalues  in order. So let us now do this. Define the index function $\iota(m,k)$ to be the smallest possible index of the Robin eigenvalue $\mu_{m,k}=x_{m,k}^2$, with a similar definition for $\iota_D(m,k)$. The words ``smallest possible" are required because these eigenvalues have multiplicity, but nothing in the subsequent work would change even if different choices were made for each pair $(m,k)$.

The following technical lemma is key to the proof of Theorem \ref{ball Pleijel constant}.
\begin{lemma}\label{lem:technical}
    Suppose that $(m_i,k_i)$ is a sequence in $\mathbb Z_{\geq 0}\times\mathbb N$. Then
    \[\lim_{i\to\infty}\frac{\iota(m_i,k_i)}{\iota_D(m_i,k_i)} = 1.\]
\end{lemma}

We defer the proof for now and instead use  Lemma \ref{lem:technical} to prove Theorem \ref{ball Pleijel constant}. 
\begin{proof}[{Proof of Theorem \ref{ball Pleijel constant}}] By  Assumption \ref{assumption:bourget}, all Robin eigenfunctions have the form
\[f_m(\theta)J(x_{\nu,k}r)\]
for some pair $(m,k)$, with $f_m(\theta)$ being a spherical harmonic (an element of $Y_m(S^{d-1})$). The same is true for Dirichlet eigenfunctions with $x_{\nu,k}$ replaced by $j_{\nu,k}$.
Consider a sequence of Robin eigenfunctions $u_i$, corresponding to eigenvalues $\mu_{\ell_i}$, for which
\[\lim_{i\to\infty}\frac{N(u_{i})}{\ell_i} = \alpha. \]By the previous paragraph, there is a sequence $(m_i,k_i)$ and a sequence of elements $f_{m_i}\in Y_{m_i}(S^{d-1})$ for which
\[u_i(\theta,r) =f_{m_i}(\theta)J(x_{\nu_i,k_i}r). \]
With this notation, $\ell_i=\iota(m_i,k_i)$.
Now consider the corresponding sequence of Dirichlet eigenfunctions
\[w_{i}(\theta,r) := f_{m_i}(\theta)J(j_{\nu_i,k_i}r).\]
We know that for each $i$, $N(u_i)=N(w_i)$, as both are equal to the number of nodal domains of $f_{m_i}(\theta)$ times $k_i$. And the index of $u_i$ as a Dirichlet eigenfunction is precisely $\iota_D(m_i,k_i)$. By Lemma \ref{lem:technical},
\[\lim_{i\to\infty}\frac{N(w_{i})}{\iota_D(m_i,k_i)}=\lim_{j\to\infty}\frac{N(u_i)}{\iota(m_i,k_i)}\frac{\iota(m_i,k_i)}{\iota_D(m_i,k_i)} = \alpha\cdot 1 = \alpha.\]
Thus any limit point of Robin nodal quotients is also a limit point of Dirichlet nodal quotients. Repeating the argument word for word shows that the reverse is true as well. Therefore the sets of limit points, and thus their suprema, are the same, completing the proof.
\end{proof}
\medskip

\subsection{Proof of Lemma \ref{lem:technical}}
We begin with a few preliminary observations. By Weyl's law for the Dirichlet eigenvalues, there is a nonzero constant $c_d$ such that as $j_{m,k}\to\infty$,
\begin{equation}\label{eq:idasymp}i_D(m,k)\sim c_d(j_{\nu,k})^d.\end{equation}
Additionally, the multiplicity coefficients $\kappa_{m,d}$ are polynomials of degree $d-2$ in $m$. Thus there is a nonzero constant $\tilde c_d$ such that as $N\to\infty$,
\begin{equation}\label{eq:sumofkappas}
\sum_{m=1}^{N}\kappa_{m,d} = \tilde c_dN^{d-1}+O(N^{d-2}).
\end{equation}
By combining \eqref{eq:interlacinglarge} and \eqref{eq:interlacingsmall}, we know that for each $k>1$,
\begin{equation}\label{eq:interlacing}
j_{\nu,k-1} < x_{\nu,k} < j_{\nu,k+1}.
\end{equation}
When $k=1$ we have the bound \cite[p.~335]{graham1973}
\begin{equation}\label{eq:kequalsone}
\max\{0,\sqrt{m^2-\sigma^2}\} \leq x_{\nu,1} < j_{\nu,2}.
\end{equation}
Finally we need some results on spacing of Bessel function zeroes. These well-known bounds follow immediately from, e.g., \cite[Theorem 2]{horsley17}. For all $\nu\geq 1/2$ and all $k$,
\begin{equation}\label{eq:besselspacing}j_{\nu,k+1}-j_{\nu,k}\leq j_{\nu,2}-j_{\nu,1},\end{equation}
and when $\nu<1/2$,
\begin{equation}\label{eq:spacingtinynu}
j_{\nu,k+1}-j_{\nu,k}\leq\pi.
\end{equation}

Now we prove Lemma \ref{lem:technical}. 

\begin{proof}[Proof of Lemma \ref{lem:technical}] Consider $\iota(m,k)$. We will find upper and lower bounds for it. Begin with an upper bound. Letting $\eta=\ell+d/2-1$ and using \eqref{eq:interlacing} we have
\[\iota(m,k)=N_{\sigma}+\sum_{(\ell,j):x_{\eta,j}\leq x_{\nu,k}}\kappa_{\ell,d}\quad \leq\quad  N_{\sigma}+\sum_{(\ell,j):x_{\eta,j}\leq j_{\nu,k+1}}\kappa_{\ell,d}.\]
Considering the $j=1$ and $j>1$ terms separately,
\[\iota(m,k)\leq N_{\sigma}+\sum_{\substack{\scriptstyle{\ell}\\\scriptstyle{x_{\eta,1}\leq j_{\nu,k+1}}}}\kappa_{\ell,d} + \sum_{\substack{\scriptstyle{(\ell,j)}\\\scriptstyle{ j\geq 2,\,\,x_{\eta,j}\leq j_{\nu,k+1}}}}\kappa_{\ell,d}.\]
Using \eqref{eq:kequalsone} for the middle term and \eqref{eq:interlacing} for the last term,
$$\iota(m,k)\leq N_{\sigma}+\sum_{\substack{\scriptstyle{\ell}\\\scriptstyle{\sqrt{\ell^2-\sigma^2}\leq j_{\nu,k+1}}}}\kappa_{\ell,d} + \sum_{\substack{\scriptstyle{(\ell,j)}\\\scriptstyle{ j\geq 2,\,\, j_{\eta,j-1}\leq j_{\nu,k+1}}}}\kappa_{\ell,d}.$$
Rewriting the middle term and relabeling $j-1$ as $j$ in the last term gives the upper bound
\begin{equation}\label{eq:upperbound}\iota(m,k)\leq N_{\sigma}+\sum_{\ell \leq \sqrt{j_{\nu,k+1}^2+\sigma^2}}\kappa_{\ell,d} + \iota_D(m,k+1).\end{equation}
Now divide by $\iota_D(m,k)$ and consider the limit as $x_{m,k}$ goes to infinity (which is equivalent to $j_{m,k}$ going to infinity and happens for any infinite subsequence of distinct $(m,k)$). The first term on the right of \eqref{eq:upperbound} goes to zero when divided by $\iota_D(m,k)$. The second term on the right of \eqref{eq:upperbound} is, by \eqref{eq:sumofkappas},
\[\tilde c_dj_{\nu,k+1}^{d-1} + O(j_{\nu,k+1})^{d-2}.\]
So division gives
\begin{equation}\label{eq:afterdiv}\frac{\iota(m,k)}{\iota_D(m,k)}\leq\frac{j_{\nu,k+1}^d}{j_{\nu,k}^d}  + \frac{\tilde c_dj_{\nu,k+1}^{d-1}}{c_dj_{\nu,k}^d} + o(1).\end{equation}
We claim that $j_{\nu,k+1}/j_{\nu,k}\to 1$, which would show that the right-hand side of \eqref{eq:afterdiv} is $1+o(1)$. Indeed, it is equivalent to show that the following goes to zero:
\[\frac{j_{\nu,k+1}-j_{\nu,k}}{j_{\nu,k}}.\]
For any subsequence with $\nu<1/2$, this goes to zero because the numerator is bounded by $\pi$ by \eqref{eq:spacingtinynu} and the denominator goes to infinity. For all other such subsequences, \eqref{eq:besselspacing} gives
\[\frac{j_{\nu,k+1}-j_{\nu,k}}{j_{\nu,k}}\leq \frac{j_{\nu,2}-j_{\nu,1}}{j_{\nu,1}+(k-1)\pi}.\]
But asymptotics of Bessel zeroes together with the fact that $j_{\nu,1}>\nu$ \cite[(10.21.3)]{DLMF} give
\[\frac{j_{\nu,k+1}-j_{\nu,k}}{j_{\nu,k}}\leq \frac{C\nu^{1/3}}{\nu+(k-1)\pi},\]
which goes to zero as $x_{\nu,k}$ goes to infinity, proving the claim. Applying the claim to \eqref{eq:afterdiv} yields
\begin{equation}\label{eq:upperboundfinal}\frac{\iota(m,k)}{\iota_D(m,k)}\leq 1 + o(1).\end{equation}

Now we want a lower bound for $\iota(m,k)$. Using \eqref{eq:interlacing} and the fact that $N_{\sigma}$ is non-negative shows that 
\[\iota(m,k)\geq \sum_{(\ell,j):j_{\mu,j+1}\leq x_{\nu,k}}\kappa_{\ell,d}.\]
We must now consider the cases $k > 1$ and $k=1$ separately. Assume that $k>1$. By \eqref{eq:interlacing},
\[\iota(m,k)\geq \sum_{(\ell,j):j_{\mu,j+1}\leq j_{\nu,k-1}}\kappa_{\ell,d}\geq i_D(m,k-1) - \sum_{\ell:j_{\mu,1}\leq j_{\nu,k-1}}\kappa_{\ell,d}.\]
Divide both sides by $\iota_D(m,k)$. The first term on the right goes to 1, by the same argument as in the upper bound case. The second term on the right goes to zero, as the numerator is $O(j_{\nu,k-1}^{d-1})$ while the denominator is $c_dj_{\nu,k}^d$. So we get
\begin{equation}
\label{eq:lowerboundfinal}\frac{\iota(m,k)}{\iota_D(m,k)}\geq 1 - o(1).\end{equation}
On the other hand, if $k=1$, then using the alternative lower bound \eqref{eq:kequalsone} gives
\[\iota(m,1)\geq \sum_{(\ell,j):j_{\mu,j+1}\leq \sqrt{m^2-\sigma^2}}\kappa_{\ell,d}\geq \sum_{(\ell,j):j_{\mu,j}\leq \sqrt{m^2-\sigma^2}}\kappa_{\ell,d}
- \sum_{\ell:j_{\mu,1}\leq \sqrt{m^2-\sigma^2}}\kappa_{\ell,d}.\]
As $m\to\infty$, the first term on the right is $c_dm^{d}+o(m^{d})$ and the second term is $O(m^{d-1})$. So dividing by $\iota_D(m,1)$ again gives $1 + o(1)$, showing \eqref{eq:lowerboundfinal} in this case as well. Lemma \ref{lem:technical} now follows from \eqref{eq:upperboundfinal} and \eqref{eq:lowerboundfinal}. \end{proof}

We now show that Assumption \ref{assumption:bourget} is generic.

\begin{proposition} \label{prop:genericity}
Assumption \ref{assumption:bourget} fails for at most a countable set of $\sigma\in\mathbb R$.
\end{proposition}
\begin{proof}

Consider the functions
\[f_m(z):=\frac{zJ_{m+d/2-1}'(z)}{J_{m+d/2-1}(z)}-\frac d2-1.\]
These functions are real analytic and monotone decreasing on each interval \[(j_{m+d/2-1,k},j_{m+d/2-1,k+1}),\] and approach $\pm\infty$ at the endpoints (zeroes of the corresponding Bessel functions). Suppose for contradiction that there exist $m_1$, $k_1$, $m_2$, and $k_2$ for which $x_{m_1,k_1}=x_{m_2,k_2}$ for an uncountable set of $\sigma$. Then there are an uncountable set of values $z$, in a given open, bounded interval (the intersection of the intervals for each individual function), for which $f_{m_1}(z)-f_{m_2}(z)=0$. Therefore this set of zeroes must have an accumulation point. The values $m_j+d/2-1$ are either both integers or both half-integers, and zeroes of Bessel functions of those orders do not coincide \cite[pp. 484-485]{watson}, so at both ends of the interval, $f_{m_1}(z)-f_{m_2}(z)$ approaches $\pm\infty$. Thus the accumulation point must be in the interior. But $f_{m_1}(z)-f_{m_2}(z)$ is analytic. So we must have $f_{m_1}(z)-f_{m_2}(z)=0$ on the entire interval. But this is a contradiction, as again, Bessel zeroes do not coincide. Thus $x_{m_1,k_1}=x_{m_2,k_2}$ may only happen for countably many $\sigma$. Taking the union over all possible choices of $m_1$, $k_1$, $m_2$, and $k_2$ yields the proposition.
\end{proof}

\begin{remark} Observe that even without making Assumption \ref{assumption:bourget}, there exists a \emph{basis} of Robin eigenfunctions on a Euclidean ball which are separated solutions. For that basis, the proofs in this section show that Pleijel's theorem holds with the same constant as for the Dirichlet ball.
\end{remark}

\section*{Acknowledgment}
The authors are grateful to Katie Gittins for helpful comments. A. H. is partially supported by EPSRC grant EP/T030577/1. D. S. was partially supported by a URC grant from DePaul University.

\bibliographystyle{alpha}
\bibliography{references}
\end{document}